\documentclass[12pt, reqno]{amsart}
\usepackage{amsmath, amsthm, amscd, amsfonts, amssymb, graphicx, color}
\usepackage[bookmarksnumbered, colorlinks, plainpages]{hyperref}
\hypersetup{colorlinks=true,linkcolor=red, anchorcolor=green, citecolor=cyan, urlcolor=red, filecolor=magenta, pdftoolbar=true}

\usepackage{amsfonts}
\usepackage{color,graphicx,shortvrb}
\usepackage{enumerate}
\usepackage{amsmath}
\usepackage{amssymb}
\usepackage{graphicx}

\usepackage{enumerate}
\usepackage{amsmath}
\usepackage{amssymb}

\usepackage[latin 1]{inputenc}

\usepackage{amsmath}

\newtheorem{theorem}{Theorem}[section]

\newtheorem{corollary}[theorem]{Corollary}
\newtheorem{lemma}[theorem]{Lemma}

\newtheorem{problem}[theorem]{Problem}
\newtheorem{proposition}[theorem]{Proposition}

\newtheorem{remark}[theorem]{Remark}

\usepackage{centernot}

\usepackage{tikz}
\usetikzlibrary{arrows,matrix}

\def\J#1#2#3{ \left\{ #1,#2,#3 \right\} }

\begin{document}

\title[Weak-local triple derivations]{Weak-local triple derivations on C$^*$-algebras and JB$^*$-triples}

\date{\today}

\author[M.J. Burgos, J.C. Cabello, A.M. Peralta]{Mar{\'i}a J. Burgos, Juan Carlos Cabello, Antonio M. Peralta}

\address{Campus de Jerez, Facultad de Ciencias Sociales y de la Comunicaci\'on
Av. de la Universidad s/n, 11405 Jerez, C\'adiz, Spain}
\email{maria.burgos@uca.es}

\address{Departamento de An{\'a}lisis Matem{\'a}tico, Universidad de Granada,
Facultad de Ciencias 18071, Granada, Spain}
\email{jcabello@ugr.es, aperalta@ugr.es}

\thanks{Authors partially supported by the Spanish Ministry of Economy and Competitiveness and European Regional Development Fund project no. MTM2014-58984-P and Junta de Andaluc\'{\i}a grants FQM375 and FQM290.}

\keywords{C$^*$-algebra; JB$^*$-triple; triple derivation; weak-local triple derivation; weak-local derivation}

\subjclass[2010]{47B47; 46L57; 17C65; 47C15; 46L05; 46L08}

\begin{abstract} We prove that every weak-local triple derivation on a JB$^*$-triple $E$ (i.e. a linear map $T: E\to E$ such that for each $\phi \in E^*$ and each $a\in E$, there exists a triple derivation $\delta_{a,\phi} : E\to E$, depending on $\phi$ and $a$, such that $\phi T(a) = \phi \delta_{a,\phi} (a)$) is a (continuous) triple derivation. We also prove that conditions \begin{enumerate}[$(h1)$]\item $\{a,{T(b)},c\}=0$ for every $a,b,c$ in $E$ with $a, c\perp b$;
\item $P_2 (e) T(a) = -Q(e) T(a)$ for every norm-one element $a$ in $E$, and every tripotent $e$ in $E^{**}$ such that $e\leq s(a)$ in $E_2^{**} (e)$, where $s(a)$ is the support tripotent of $a$ in $E^{**}$,
\end{enumerate} are necessary and sufficient to show that a linear map $T$ on a JB$^*$-triple $E$ is a triple derivation.
\end{abstract}

\maketitle

\section{Introduction}

A \emph{triple derivation} on a JB$^*$-triple $E$ is a linear mapping $\delta: E\to E$ satisfying the following Leibniz' rule \begin{equation}\label{eq triple der}\delta\{a,b,c\} = \{\delta(a), b,c\} +\{a,\delta(b),c\}+\{a,b,\delta(c)\},
\end{equation} for every $a,b,c\in E$. T. Barton and Y. Friedman proved, in \cite[Corollary 2.2]{BarFri}, that the geometric structure of JB$^*$-triples assures that every triple derivation on a JB$^*$-triple.\smallskip

Motivated by the different studies on local (associative) derivations on C$^*$-algebras (compare \cite{Kad90, Shu, John01}), M. Mackey introduced and presented, in \cite{Mack}, the first study on local triple derivations on JB$^*$-triples. We recall that a \emph{local triple derivation} on a JB$^*$-triple $E$ is a linear map $T : E\to E$ such that for each $a$ in $E$ there exists a triple derivation $\delta_{a}: E\to E$, depending on $a$, satisfying $T(a) = \delta_a (a).$ It is due to Mackey that every continuous local triple on a JBW$^*$-triple (i.e. a JB$^*$-triple which is also a dual Banach space) is a triple derivation (see \cite[Theorem 5.11]{Mack}). The first and third author of this note, in collaboration with F.J. Fern{\'a}ndez-Polo, establish in \cite{BurPolPerBLMS14} that every local triple derivation on a JB$^*$-triple is continuous and a triple derivation.\smallskip

In the setting of C$^*$-algebras, B.A. Essaleh, M.I. Ram{\'i}rez, and the third author of this note explore the notions of weak-local derivation and weak$^*$-local derivation on C$^*$-algebras and von Neumann algebras, respectively (see \cite{EssaPeRa16,EssaPeRa16b}). Going back in history, we remind that, according to Kadison's definition, a \emph{local derivation} from a C$^*$-algebra $A$ into a Banach $A$-bimodule, $X$, is a linear map $T: A\to X$ such that for each $a\in A$ there exits a derivation $D_{a}: A\to X$, depending on $a$, satisfying $T(a) = D_a(a)$. B.E. Johnson proved in \cite{John01} that every local derivation from a C$^*$-algebra $A$ into a Banach $A$-bimodule is continuous and a derivation. Following \cite{EssaPeRa16}, a linear mapping $T: A\to X$ is called a \emph{weak-local derivation} if for each $\phi \in X^*$ and each $a\in A$, there exists a derivation $D_{a,\phi} : A\to X$, depending on $\phi$ and $a$, such that $\phi T(a) = \phi D_{a,\phi} (a)$. It is shown in \cite[Theorems 2.1 and 3.4]{EssaPeRa16} (see also \cite{EssaPeRa16b}) that every weak-local derivation on a C$^*$-algebra is continuous and a derivation. Similarly, if $W$ is a von Neumann algebra (i.e. a C$^*$-algebra which is also a dual Banach space) a \emph{weak$^*$-local derivation} on $W$ is a linear map $T: W\to W$ such that for each $\phi\in W_*$ and each $a\in W$, there exists a derivation $D_{a,\phi} : W\to W$, depending on $\phi$ and $a$, such that $\phi T(a) = \phi D_{a,\phi} (a)$. Weak$^*$-local derivations on a von Neumann algebra are automatically continuous and derivations (see \cite[Theorems 2.8 and 3.1]{EssaPeRa16}).\smallskip

In the wider setting of JB$^*$-triples, weak-local triple derivations seem a natural notion to explore in this line. We shall say that a linear map $T$ on a JB$^*$-triple $E$ is a \emph{weak-local triple derivation} if for each $\phi \in E^*$ and each $a\in E$, there exists a triple derivation $\delta_{a,\phi} : E\to E$, depending on $\phi$ and $a$, such that $\phi T(a) = \phi \delta_{a,\phi} (a)$. \emph{Weak$^*$-local triple derivations} on a JBW$^*$-triple are similarly defined.\smallskip

In this note we prove that every weak-local triple derivation on a JB$^*$-triple is continuous and a triple derivation (see Theorems \ref{t continuity of weak-local triple derivations} and \ref{t weak-local derivations}). 
The proof of the main result will be derived with appropriate generalizations of the technical results stated in \cite{BurPolPerBLMS14} for local triple derivations. Among the new results obtained in the study of the weak-local characterization of triple derivations, we also obtain that conditions \begin{enumerate}[$(h1)$]\item $\{a,{T(b)},c\}=0$ for every $a,b,c$ in $E$ with $a, c\perp b$;
\item $P_2 (e) T(a) = -Q(e) T(a)$ for every norm-one element $a$ in $E$, and every tripotent $e$ in $E^{**}$ such that $e\leq s(a)$ in $E_2^{**} (e)$, where $s(a)$ is the support tripotent of $a$ in $E^{**}$,
\end{enumerate} are necessary and sufficient to show that a linear map $T$ on a JB$^*$-triple $E$ is a triple derivation (see Theorem \ref{t weak-local derivations}). Actually, $(h2)$ is enough to guarantee that $T$ is continuous (see Proposition \ref{p second necessary condition for weaklocal der implies continuity}).

\section{The results}

The class $\mathcal{J}$ of those complex Banach spaces whose open unit ball is a bounded symmetric domain strictly includes all C$^*$-algebras (see \cite{Harris74}) and all complex Hilbert spaces. There is a undoubted advantage in determining elements in the bigger class by a list of geometric and algebraic axioms. The characterization in precise axiomatic terms of those elements in $\mathcal{J}$ is due to W. Kaup (see \cite{Ka83}), who proved that every element in $\mathcal{J}$ is a \emph{JB$^*$-triple}, that is, a complex Banach space $E$ admitting a continuous triple product $\{.,.,.\} : E\times E\times E \to E$, $(x,y,z)\mapsto \{x,y,z\},$ which is linear and symmetric in $x$ and $z$ and conjugate linear in $y$, satisfies the so-called Jordan identity
$$\{x,y,\{a,b,c\}\} = \{\{x,y,a\},b,c\} - \{a,\{y,x,b\},c\} + \{a,b,\{x,y,c\}\},$$ for every $x,y,a,b,c\in E$, and, for each $a\in E$, the operator $x\mapsto L(a,a) (x):= \{a,a,x\}$ is hermitian with non-negative spectrum and $\|L(a,a)\| = \|a\|^2$.\smallskip

Every C$^*$-algebra $A$ is a JB$^*$-triple when it is equipped with the product $\{a,b,c\} = \frac12 (a b^* c + c b^* a)$. The category of JB$^*$-triples is strictly bigger and contains, for example, all complex Hilbert spaces and all JB$^*$-algebras.\smallskip

JB$^*$-triples whose underlying Banach space is a dual space are called \emph{JBW$^*$-triples}. In particular, every von Neumann algebra is a JBW$^*$-triple. The second dual $E^{**}$ of a JB$^*$-triple $E$ is JBW$^*$-triple (cf. \cite[Corollary 3.3.5]{Chu2012}). Every JBW$^*$-triple $W$ admits a unique isometric predual $W_*$, and its triple product is separately $\sigma(W,W_*)$-continuous (cf. \cite{BarTi}). We refer to the monograph \cite{Chu2012} for the notions not included in this paper.\smallskip

Let $E$ be a JB$^*$-triple. It is known that the bitranspose, $\delta^{**}: E^{**}\to E^{**},$ of a triple derivation $\delta : E\to E$ is a triple derivation (compare \cite[Statement  (4) in page 713]{BurPolPerBLMS14}).

\subsection{Continuity of weak-local triple derivations}\ \ \newline

We have already commented that T.J. Barton and Y. Friedman conducted the first result on automatic continuity of triple derivations on an arbitrary JB$^*$-triple. The core of their arguments appears in \cite[Theorem 2.1]{BarFri}, where they establish that every triple derivation $\delta$ on a JB$^*$-triple $E$ is dissipative, i.e., for every $\phi\in E^*$ and every $x\in E$ with $\phi(x) = 1=\|x\|=\|\phi\|$, we have $\Re\hbox{e} \phi\delta(x)\leq 0$ (actually, $\Re\hbox{e} \phi\delta(x)= 0$, for every $\phi$ and $x$ as above). It can be therefore deduced, by the theory of dissipative operators, that every triple derivation on a JB$^*$-triple is continuous (compare \cite[Proposition 3.1.15]{BraRo}). This is enough to show that every local triple derivation on a JB$^*$-triple is dissipative and hence continuous (see \cite[Theorem 2.8]{BurPolPerBLMS14}). We can actually see that the optimal notion to inherit the dissipative nature of triple derivations is the notion of weak-local triple derivation. Namely, let $T: E\to E$ be a weak-local triple derivation on a JB$^*$-triple. Suppose we take $\phi\in E^*$ and $x\in E$ with $\phi(x) = 1=\|x\|=\|\phi\|,$ then $$\Re\hbox{e} \phi T(x) = \Re\hbox{e} \phi \delta_{x,\phi} (x)\leq 0,$$ which shows that $T$ is dissipative.

\begin{theorem}\label{t continuity of weak-local triple derivations} Every weak-local triple derivation on a JB$^*$-triple is dissipative and hence continuous.
\end{theorem}

The theory of dissipative maps doesn't work when we are dealing with weak$^*$-local triple derivation on a JBW$^*$-triple. We do not know the answer to the next problem.

\begin{problem}\label{problem continuity of weak*-local} Is every weak$^*$-local triple derivation on a JBW$^*$-triple automatically continuous?
\end{problem}

More results on automatic continuity of maps related to weak-local triple derivations will be discussed in the next subsection.

\subsection{Weak local triple derivations are triple derivations}\ \ \newline

The first statement in the next result is a consequence of the Hahn-Banach theorem, while the second is clear. This simple result is included to simplify the arguments in the subsequent results, and the proof is left to the reader.

\begin{lemma}\label{l projections} Let $P: X\to X$ be a real linear projection on a Banach space. Suppose $x$ is an element in $X$
satisfying that for every $\phi \in X^*$ with $P^{*} (\phi) = \phi P = \phi$ we have $\phi (x) =0$. Then $P(x)=0$. If $\phi$ is a functional in $X^*$ such that
$\phi (a)=0$, for every $a= P(a) \in X$, then $P^{*} (\phi) =0$.$\hfill\Box$
\end{lemma}

Tripotents in a JB$^*$-triple play a role similar to that of projections in a C$^*$-algebra. An element $e$ in a JB$^*$-triple $E$ is said to be a \emph{tripotent} if $\{e,e,e\} =e$. Every tripotent determines a decomposition of $E$ in terms of the eigenspaces of the operator $L(e,e)$. More precisely, $$E= E_0 (e) \oplus E_1(e) \oplus E_2 (e),$$ where $E_j (e) = \{ x\in E : L(e,e) (x)= \frac{j}{2} x\}$. This decomposition is called the \emph{Peirce decomposition} of $E$ relative to $e$. These subspaces satisfy the following Peirce rules $$\J {E_{i}(e)}{E_{j} (e)}{E_{k} (e)}\subseteq E_{i-j+k} (e)$$ if $i-j+k \in \{ 0,1,2\},$ and $\J {E_{i}(e)}{E_{j} (e)}{E_{k} (e)}=\{0\}$ otherwise. Moreover, $\J {E_{2} (e)}{E_{0}(e)}{E} = \J {E_{0} (e)}{E_{2}(e)}{E} =\{0\}.$ The corresponding \emph{Peirce projections}, $P_{i} (e) : E\to E_{i} (e)$, $(i=0,1,2)$ are given by $$P_2 (e) = Q(e)^2, \ P_1 (e) = 2 L(e,e) - 2 Q (e)^2, \ \hbox{ and } P_0 (e) = Id-2 L(e,e) + Q (e)^2,$$ where $Id$ is the identity map on $E$, and $Q(e)$ is the conjugate linear operator on $E$ defined by $Q(e) (x) :=\{e,x,e\}$. Clearly, $L(e,e) = P_2 (e) +\frac12 P_1 (e)$.  If $E_2 (e) $ reduces to $\mathbb{C} e$ we say that $e$ is minimal.\smallskip

The separate weak$^*$-continuity of the triple product in a JBW$^*$-triple $M$ implies that Peirce projections associated with a tripotent $e$ in $M$ are weak$^*$-continuous.\smallskip

The Peirce subspace $E_2 (e)$ enjoys an additional structure. Namely, the Jordan product and involution defined by $x\circ_e y := \J xey$ and involution $x^{*_e} := Q(e)(x)$, respectively, are well-defined on $E_2(e)$ and equip the latter space with a structure of unital JB$^*$-algebra with unit $e$.\smallskip

Let us recall some definitions. A central notion in the study of triple derivations is the relation of orthogonality. Elements $a,b$ in a JB$^*$-triple $E$ are said to be \emph{orthogonal} (written $a\perp b$) if $L(a,b) =0$. It is known that any of the following conditions is equivalent to $a\perp b$
\begin{center}
\begin{tabular}{ccc}\label{ref orthogo}
  $b\perp a;$ & $\J aab =0;$ & $a \perp r(b);$  \\
  & & \\
  $r(a) \perp r(b);$ & $E^{**}_2(r(a)) \perp E^{**}_2(r(b))$; & $r(a) \in E^{**}_0 (r(b))$; \\
  & & \\
  $a \in E^{**}_0 (r(b))$; & $b \in E^{**}_0 (r(a))$; & $E_a \perp E_b$, \\
\end{tabular}
\end{center} where $E_a$ denotes the JB$^*$-subtriple of $E$ generated by $a$
(compare \cite[Lemma 1]{BurFerGarMarPe}). \smallskip

Let $\delta: E\to E$ be a triple derivation on a JB$^*$-triple. Suppose $e$ is a tripotent in $E$. As observed in \cite{Mack} and \cite{BurPolPerBLMS14}, the identity $\delta(e) = \delta\{e,e,e\} = 2\{\delta(e),e,e\} +\{e,\delta(e),e\}$ shows that \begin{equation}\label{eq derivation at a tripotent} P_0 (e) \delta (e) = 0 \hbox{ and } P_2 (e) \delta(e) =-Q(e) \delta (e).
\end{equation}\smallskip

In the next result we study the behavior of a weak-local triple derivation at a tripotent.

\begin{proposition}\label{p weak-local triple derivation at a tripotent}
Let $T:E\to E$ be a weak-local triple derivation on a JB$^*$-triple, and let $e$ be a tripotent in $E$. Then the following statements hold:
\begin{enumerate}[$(a)$]\item $P_0 (e) T(e) = 0$;
\item $P_2 (e) T(e) = -Q(e) T(e)$;
\item $T(e) = 2\{T(e),e,e\} +\{e,T(e),e\}.$
\end{enumerate}
\end{proposition}

\begin{proof} $(a)$ Let $\phi$ be an element in $E^*$ satisfying $\phi = \phi P_0 (e)$. By our assumptions, $$\phi T(e)= \phi \delta_{e,\phi} (e) = \phi P_0 (e) \delta_{e,\phi} (e) = 0,$$ where in the last equality we apply \eqref{eq derivation at a tripotent}. It follows from Lemma \ref{l projections} that $P_0 (e) T(e) =0$.\smallskip

$(b)$ Let $\mathcal{Q}= \frac12 (P_2(e) +Q(e))$. Clearly, $\mathcal{Q}$ is a contractive real linear projection on $E$. Pick a functional $\phi\in E^*$ with $\phi = \phi \mathcal{Q}$. The hypothesis combined with \eqref{eq derivation at a tripotent} show that $$\phi T(e) = \phi \delta_{e,\phi} (e) = \phi \mathcal{Q} \delta_{e,\phi} (e) = 0.$$ Lemma \ref{l projections} gives $\mathcal{Q} T(e) = 0$, which proves the desired statement. \smallskip

$(c)$ Applying $(a)$ and $(b)$ and the Peirce decomposition we get $$T(e) = P_0 (e) T(e) + P_1 (e) T(e) + P_2 (e) T(e) = P_1 (e) T(e) + P_2 (e) T(e) $$ $$=2 L(e,e) T(e) - P_2 (e) T(e) = 2\{e,e,T(e)\} + Q(e) T(e) = 2\{e,e,T(e)\} + \{e, T(e),e\}.$$
\end{proof}

\begin{remark}\label{remark prop 1 for weak*-local} The conclusions in the above Proposition \ref{p weak-local triple derivation at a tripotent} remain true when we assume that $T:E\to E$ is a weak$^*$-local triple derivation on a JBW$^*$-triple.
\end{remark}

A closer look at the arguments given in the first part of \cite[Proof of Theorem 2.4]{BurPolPerBLMS14} allows us to realize that the next result follows from the same arguments. We shall include here an sketch of the proof for completeness reasons.

\begin{proposition}\label{p weak local for algebraic} Let $T: E\to E$ be a linear mapping on a JB$^*$-triple. Suppose that $T(e) = 2\{e,e,T(e)\} + \{e, T(e),e\}$ for every tripotent $e$ in $E$. Then $T\{a,a,a\} = 2\{T(a),a,a\} + \{a, T(a),a\},$ whenever $a$ writes as a finite linear combination of mutually orthogonal tripotents in $E$.
\end{proposition}

\begin{proof} Let $e_1,$ $e_2,\ldots, e_n$ be mutually orthogonal tripotents in $E$. Proposition \ref{p weak-local triple derivation at a tripotent}$(a)$ implies that $P_0 (e_j) T (e_j)=0.$ The condition $e_i,e_k\perp e_j$ implies that $e_i,e_k\in E_0 (e_j)$ and hence, by Peirce arithmetic, we have \begin{equation}\label{eq 1 prop algebraic} \{ {e_i},{T(e_j)},{e_k} \}= 0.
\end{equation} By hypothesis, \begin{equation}\label{eq 2 prop algebraic} T \J {e_i }{e_i}{e_i} = 2 \J {e_i }{e_i }{T (e_i)} + \J {e_i }{T(e_i)}{e_i},
\end{equation} for every $i$. Furthermore, for each $i\neq j$ in $\{1,\ldots, n\}$, since $e_i\perp e_j$, and hence $e_i\pm e_j$ is a tripotent, it follows from the hypothesis that $$ T \J {e_i \pm e_j}{e_i \pm e_j}{e_i \pm e_j} = 2 \J {e_i \pm e_j}{e_i \pm e_j}{T (e_i \pm e_j)}$$ $$ + \J {e_i \pm e_j}{T(e_i \pm e_j)}{e_i \pm e_j}.$$ Expanding the above identity and having in mind \eqref{eq 2 prop algebraic} and $e_i\perp e_j$, we deduce that $$\pm 2 \J {e_i}{e_i}{T (e_j)} + 2 \J {e_j}{e_j}{T (e_i)} \pm  \J {e_i }{T(e_j)}{e_i )} + \J {e_j}{T(e_i)}{ e_j} $$ $$\pm 2 \J {e_i}{T(e_i)}{ e_j} + 2 \J {e_i}{T( e_j)}{e_j} =0, $$ which shows that $$+ 4 \J {e_j}{e_j}{T (e_i)} + 2 \J {e_j}{T(e_i)}{ e_j}+ 4 \J {e_i}{T( e_j)}{e_j} =0.$$ The identity proved in \eqref{eq 1 prop algebraic} assures that \begin{equation}\label{eq 3 proposition algebraic} \J {e_j}{e_j}{T (e_i)} + \J {e_i}{T( e_j)}{e_j} =0.
\end{equation}

Suppose $a$ is an element in $E$ which can be written in the form $\displaystyle a=\sum_{i=1}^{n} \lambda_i e_i,$ where $e_1 , \ldots, e_n$ are mutually orthogonal tripotents. By the linearity of $T$, expanding the expressions $$\displaystyle T\{a,a,a\} = \sum_{i=1}^{n} \lambda_i^{3} T (\J {e_i}{e_i}{e_i}),\ \displaystyle 2 \J {T(a)}aa =2 \sum_{i,j=1}^{n} \lambda_i^2 \lambda_j \J {e_i}{e_i}{T (e_j)},$$ and $\displaystyle \J {a}{T(a)}{a} = \J {\sum_{i=1}^{n} \lambda_i e_i}{\sum_{j=1}^{n} \lambda_j T(e_j)}{\sum_{k=1}^{n} \lambda_k e_k}$, it can be easily checked, applying \eqref{eq 1 prop algebraic}, \eqref{eq 2 prop algebraic} and \eqref{eq 3 proposition algebraic}, that $$T \J aaa = 2 \J {T(a)}aa + \J a{T(a)}a.$$
\end{proof}

Unfortunately, there exist examples of JB$^*$-triples lacking of tripotents. Actually, every C$^*$-algebra lacking of projections is an example of the previous statement. In a JBW$^*$-triple $M$, every element in $M$ can be approximated in norm by finite linear combinations of mutually orthogonal tripotents in $M$ (see \cite[Lemma 3.11]{Ho87}). In a weakly compact JB$^*$-triple in the sense of \cite{BuChu}, every element can be approximated in norm by finite linear combinations of mutually orthogonal minimal tripotents (see \cite[Remark 4.6]{BuChu}).\smallskip

We can prove now that the aforementioned Mackey's theorem remains valid for weak-local triple derivations on JBW$^*$-triples and on weakly compact JB$^*$-triples.

\begin{corollary}\label{c weak local when algebraic are dense} Let $E$ be a JB$^*$-triple satisfying that every element in $E$ can be approximated in norm by finite linear combinations of mutually orthogonal tripotents in $E$. 
Let $T: E\to E$ be a bounded linear operator satisfying $$T(e) = 2\{e,e,T(e)\} + \{e, T(e),e\},$$ for every tripotent $e$ in $E$. Then $T$ is a triple derivation. Consequently, every weak-local triple derivation on $E$ is a triple derivation.
\end{corollary}

\begin{proof} By Proposition \ref{p weak local for algebraic} we know that $T\{a,a,a\} = 2\{T(a),a,a\} + \{a, T(a),a\},$ whenever $a$ writes as a finite linear combination of mutually orthogonal tripotents in $E$. Since the latter elements are norm-dense in $E$, $T$ is continuous, and the triple product of $E$ is jointly norm continuous, we obtain that $T\{b,b,b\} = 2\{T(b),b,b\} + \{b, T(b),b\},$ for every $b\in E$. A standard polarisation argument, via the formula: \begin{equation}
\label{polarization fla} \{ x,y,z\} = 8^{-1} \sum_{k=0}^{3} \sum_{j=1}^{2} (-1)^{j} i^{k} \Big(x+ i^{k} \ y + (-1)^{j} z\Big)^{[3]} \ \ \
(x,y, z\in {E}),
\end{equation} proves that $T$ is a triple derivation (see, for example, the proof of \cite[Theorem 5.11]{Mack}).\smallskip

For the last statement, let $S: E\to E$ be a weak-local triple derivation. We observe that Theorem \ref{t continuity of weak-local triple derivations} implies that $S$ is continuous.  By Proposition \ref{p weak-local triple derivation at a tripotent}$(c)$, $S\{e,e,e\} = 2 \{S(e),e,e\}+\{e, S(e),e\}$, for every tripotent $e\in E$. The desired conclusion follows from the first part of this corollary.
\end{proof}

Combining Proposition \ref{p weak local for algebraic} with the appropriate weak$^*$-local version of Proposition \ref{p weak-local triple derivation at a tripotent} given in Remark \ref{remark prop 1 for weak*-local} we obtain:

\begin{corollary}\label{c continuous weak*-local triple deriv} Every continuous weak$^*$-local triple derivation on a JBW$^*$-triple is a triple derivation.
\end{corollary}

As we have previously mentioned, a general JB$^*$-triple $E$ might not contain a single tripotent. For these reasons, as in the study of local triple derivations, we must deal with tripotents in $E^{**}$ which are achievable by elements in $E$.\smallskip

We recall a series of definitions and notions on range and compact tripotents taken from \cite{EdRu96,FerPe06,FerPe07}. Let $a$ be a norm-one element in a JB$^*$-triple $E$. We set $a^{[1]} = a$ and $a^{[2 n +1]} := \J a{a^{[2n-1]}}a$ $(\forall n\in \mathbb{N})$. It is also known that the sequence $(a^{[2n -1]})$ converges in the
weak$^*$ topology of $E^{**}$ to a unique tripotent $s(a)$ in $E^{**}$. The tripotent $s(a)$ is called the \emph{support} \emph{tripotent} of $a$ in $E^{**}$. A tripotent $e$ in $E^{**}$ is said to be \emph{compact-$G_{\delta}$} (relative to $E$) if there exists a norm-one element $a$ in $E$ such that $e$ coincides with $s(a)$. A tripotent $e$ in $E^{**}$ is said to be \emph{compact} (relative to $E$) if there exists a decreasing net $(e_{\lambda})$ of tripotents in $E^{**}$ which are compact-$G_{\delta}$ with infimum $e$, or if $e$ is zero (compare \cite{EdRu96}).\smallskip

The JB$^*$-subtriple $E_a$ generated by a single element $a$ in a JB$^*$-triple $E$ can be identified isometrically with a commutative C$^*$-algebra of form $C_0(L)$ where $L\subseteq [0,\|a\|]$ with $L\cup\{0\}$ compact, in such a way that the element $a$ is associated with a positive generating element in $C_0(L)$ (compare \cite[\S 1, Corollary 1.15]{Ka83}). A continuous triple functional calculus at the element $a$ can be appropriately defined. Consequently, for each natural $n$, there exists (a unique) element $a^{[{1}/({2n-1})]}$ in $E_a$ satisfying $(a^{[{1}/({2n-1})]})^{[2n-1]} = a.$ The sequence $(a^{[{1}/({2n-1})]})$ converges in the weak$^*$ topology of $E^{**}$ to a tripotent (called the \emph{range tripotent} of $a$) which is denoted by $r(a)$. The tripotent $r(a)$ can be also defined as the smallest tripotent $e$ in $E^{**}$ satisfying that $a$ is positive in the JBW$^*$-algebra $E^{**}_{2} (e)$. It is also known that $ s(a) \leq a \leq r(a)$ in $E^{**}_2 (r(a))$.\smallskip

We shall state next a variant of the first identity in \eqref{eq derivation at a tripotent}. Let $\delta: E\to E$ be a triple derivation on a JB$^*$-triple. Let $c$ be an element in $E$, and let $e$ be a tripotent in $E^{**}$ such that $c\in E_{2}^{**} (e)$. We claim that \begin{equation}\label{eq triple derivation at an element and ranges 1} P_0 (e) \delta(c) =0.
\end{equation} Indeed, by the continuous triple functional calculus (compare \cite[\S 1, Corollary 1.15]{Ka83}), there exists $z\in E\cap E_{2}^{**} (e)$ such that $\{z,z,z\} = c$. Since, $\delta (c) = 2\{\delta(z),z,z\} + \{z,\delta(z),z\}$, it follows from Peirce rules that $$\{z,\delta(z),z\}\in E\cap E^{**}_2 (e), \hbox{ and } \{\delta(z),z,z\} \in E\cap \left(E^{**}_2 (e) \oplus E^{**}_1 (e) \right),$$ and hence   $\delta (c) \in E\cap \left(E^{**}_2 (e) \oplus E^{**}_1 (e) \right),$ witnessing the desired conclusion in \eqref{eq triple derivation at an element and ranges 1}.

\begin{lemma}\label{l weak-local triple deriv range} Let $T: E\to E$ be a weak-local triple derivation on a JB$^*$-triple. Let $a$ be an element in $E$, and let $e$ be a tripotent in $E^{**}$ such that $a\in E_2^{**} (e)$. Then $P_0 (e) T(a) =0$.
\end{lemma}

\begin{proof} Let $\phi\in E^{*}$ be a functional satisfying $\phi = \phi P_0 (e)$. It follows from the hypothesis and \eqref{eq triple derivation at an element and ranges 1} that $$\phi T(a) = \phi \delta_{a,\phi} (a) = \phi P_0 (e) \delta_{a,\phi} (a) = 0.$$ Lemma \ref{l projections} assures that $P_0 (e) T(a) =0$.
\end{proof}

We can derive now a weak-local version of \cite[Lemma 4]{BurFerGarPe2012} with a simple argument.

\begin{lemma}\label{l 4 BurFerGarPe} Let $T: E \to E$ be a weak-local triple derivation on a JB$^*$-triple. Then $\{a,{T(b)},c\}=0$ for every $a,b,c$ in $E$ with $a, c\perp b$.
\end{lemma}

\begin{proof} Let $a,b,c$ in $E$ with $a,c\perp b$. Applying Lemma \ref{l weak-local triple deriv range} with $e= r(b)$ we deduce that $P_0 (r(b)) T(b) =0$. It follows from the hypothesis that $a,c\in E_0^{**}(r(b))$. So, we can easily deduce, via Peirce arithmetic, that $\{a,{T(b)},c\}=0$.
\end{proof}

A detailed inspection to the proof of \cite[Proposition 2.2]{BurPolPerBLMS14} is enough to ensure that the conclusions in the first two statements of the just quoted Proposition 2.2 in \cite{BurPolPerBLMS14} remain valid when $T:E\to E$ is a bounded linear operator on a JB$^*$-triple satisfying $\{a,T(b),c\}=0$, for every $a,b,c\in E$ with $a,c\perp b$. We therefore have:

\begin{proposition}\label{p local triple derivations on compact tripotents via orthognal forms}\cite[proof of Proposition 2.2]{BurPolPerBLMS14} Let $T : E \to E$ be a bounded linear operator on a JB$^*$-triple satisfying $\{a,T(b),c\}=0$, for every $a,b,c\in E$ with $a,c\perp b$. Suppose $e$ is a compact tripotent in $E^{**}.$ Then the following statements hold: \begin{enumerate}[$(a)$]
\item $P_0 (e) T^{**}(e) =0$;
\item If $a$ is a norm-one element in $E$ whose support tripotent is $e$ {\rm(}that is, $e$ is a compact-$G_{\delta}$ tripotent{\rm)}, then $Q (e) T(a) = Q (e) T^{**} (e).$
\end{enumerate}
\end{proposition}

\begin{remark}\label{remark orthogonal forms is not enough for triple derivation} We observe that a mapping satisfying the hypothesis of the above Proposition \ref{p local triple derivations on compact tripotents via orthognal forms} need not be, in general a triple derivation. Take, for example a unital C$^*$-algebra $A$, a non-normal element $x_0\in A$, and the mapping $T: A\to A$ defined by $T(a) :=[x_0,a] = x_0 a - a x_0$. Since $T\{1,x_0,1\} = T(x_0^*) = x_0 x_0^* - x_0^* x_0 \neq 0$, and $2\{T(1), x_0, 1\} + \{1,T(x_0),1\} = 0,$ we deduce that $T$ is not a triple derivation. However, given $a,b,c\in A$ with $a,c\perp b$ {\rm(}i.e. $a b^* = b^* a= cb^* = b^* c=0${\rm)} we have $$2 \{a,T(b),c\} = a [x_0,b]^* c + c [x_0,b]^* a = 0.$$
\end{remark}

Given two tripotents $e$ and $u$ in a JB$^*$-triple $E$ we write $u \leq e$ if $e-u$ is a tripotent in $E$ and $e-u \perp u$.\smallskip

Our next lemma gives a condition, which added to the hypothesis in Proposition \ref{p local triple derivations on compact tripotents via orthognal forms}, avoids the difficulties appearing by the counterexample given in Remark \ref{remark orthogonal forms is not enough for triple derivation}.

\begin{lemma}\label{l weak-local triple deriv compact b} Let $T: E\to E$ be a weak-local triple derivation on a JB$^*$-triple. Let $a$ be a norm-one element in $E$, and let $e$ be a tripotent in $E^{**}$ such that $e\leq s(a)$ in $E_2^{**} (e)$, that is, a tripotent satisfying $P_2 (e) (a)= \{e,a,e\} = e$. Then $P_2 (e) T(a) = -Q(e) T(a)$  in $E^{**}$.
\end{lemma}

\begin{proof} Let $\delta : E\to E$ be a triple derivation. It is known that $\delta^{**} : E^{**}\to E^{**}$ is a triple derivation. In $E^{**}$, the element $a$ writes in the form $a = e + P_0 (e) (a)$. A new application of the continuous triple functional calculus and the fact that $E_0^{**} (e)$ is a subtriple of $E^{**}$ assures the existence of $z\in E^{**}_0 (e)$ such that $\{z,z,z\} = P_0 (e) (a)$. In particular, $\delta^{**}  P_0 (e) (a) = 2 \{\delta^{**} (z),z,z\} + \{z,\delta^{**}  (z), z\}$ lies in $E^{**}_0 (e) \oplus E_1^{**} (e)$, by Peirce arithmetic. Therefore, by \eqref{eq derivation at a tripotent} and Peirce arithmetic, we have $$Q (e) \delta (a) = Q(e) \delta^{**} (e) +  Q(e) \delta^{**}  P_0 (e) (a) = Q(e) \delta^{**} (e) = -P_2 (e) \delta^{**} (e),$$ which implies that $P_2 (e) \delta (a) = -Q (e) \delta^{**} (e).$\smallskip

Now take $\phi \in E^{*}$ satisfying $\phi = \phi \frac12( P_2 (e) + Q(e))$. It follows from the hypothesis and the above arguments that $$\phi T(a) = \phi \delta_{a,\phi} (a) = \phi \frac12( P_2 (e) + Q(e)) \delta_{a,\phi} (a) = \phi \frac12( P_2 (e) + Q(e)) \delta_{a,\phi}^{**} (e) = 0,$$ where the last equality follows from \eqref{eq derivation at a tripotent}. Lemma \ref{l projections} implies that $P_2 (e) T(a) = -Q(e) T(a)$.
\end{proof}

Let $\varphi$ be a norm-one functional in the predual of a JBW$^*$-triple $W$. B. Russo and Y. Friedman prove in \cite[Proposition\ 2]{FriRu85} the existence of a unique tripotent $e= e(\varphi) \in W$ satisfying $\varphi = \varphi P_{2} (e),$ and $\varphi|_{W_{2}(e)}$ is a faithful normal
state of the JBW$^*$-algebra $W_{2} (e)$. This unique tripotent $e$ is called the \emph{support tripotent} of
$\varphi$.\smallskip

We shall prove next that the property obtained in the conclusion of the above Lemma \ref{l weak-local triple deriv compact b} is enough to guarantee the automatic continuity of a linear mapping.

\begin{proposition}\label{p second necessary condition for weaklocal der implies continuity} Let $T: E\to E$ be a linear mapping on a JB$^*$-triple. Suppose that for every norm-one element $a$ in $E$, and every tripotent $e$ in $E^{**}$ such that $e\leq s(a)$ in $E_2^{**} (e)$ we have $P_2 (e) T(a) = -Q(e) T(a)$ in $E^{**}$. Then $T$ is continuous.
\end{proposition}

\begin{proof} Let us pick $\phi\in E^{*}$, $a\in E$ with $\|a\|=1=\|\phi\| =\phi (a)$. Let $E_a$ be the JB$^*$-subtriple of $E$ generated by $a$. Proposition 2.5 in \cite{PeExpo2015} implies that $\phi|_{E_a}$ is a triple homomorphism. Regarding $\phi$ as a normal functional in the predual of $E^{**}$, we can deduce, via weak$^*$ density that $\phi$ is a triple homomorphism when restricted to the weak$^*$-closure of $E_{a}$ in $E^{**}$. In particular, $\phi (s(a)) =1$. Let $e$ denote the support tripotent of $\phi $ in $E^{**}$. It follows from the properties defining the support tripotent that $s(a)\geq e$. The hypothesis on $T$ implies that $P_2 (e) T(a) = -Q(e) T(a)$ in $E^{**}$. Since $\phi|_{E^{**}_2 (e)}$ is a faithful normal positive functional in the JBW$^*$-algebra $(E_2^{**} (e), \circ_{e},Q(e))$, and $\phi = \phi P_2 (e)$, we have $$ \phi T(a) = \phi P_2(e) T(a) =  \phi \left(\frac{P_2(e)+Q(e)}{2} T(a)\right) =0.$$ We have therefore shown that $T$ is dissipative, and hence continuous (cf. \cite[Proposition 3.1.15]{BraRo}).
\end{proof}

Combining Lemma \ref{l 4 BurFerGarPe} with Proposition \ref{p local triple derivations on compact tripotents via orthognal forms} we can now derive now a weak-local version of \cite[Propostion 2.2]{BurPolPerBLMS14}.

\begin{proposition}\label{p weak-local triple derivations on compact tripotents new} Let $T : E \to E$ be a weak-local triple derivation on a JB$^*$-triple. Suppose $e$ is a compact tripotent in $E^{**}.$ Then the following statements hold: \begin{enumerate}[$(a)$]
\item $P_0 (e) T^{**}(e) =0$;
\item If $a$ is a norm-one element in $E$ whose support tripotent is $e$ (that is, $e$ is a compact-$G_{\delta}$ tripotent), then $Q (e) T(a) = Q (e) T^{**} (e);$
\item $P_2(e) T^{**}(e) = - Q(e) T^{**}(e);$
\item $T^{**} \{e,e,e\} = 2 \{T^{**} (e),e,e\} + \{e,T^{**} (e), e\}$.
\end{enumerate}
\end{proposition}

\begin{proof} The first two statements follow from Theorem \ref{t continuity of weak-local triple derivations}, Lemma \ref{l 4 BurFerGarPe}, and Proposition \ref{p local triple derivations on compact tripotents via orthognal forms}. To prove $(c)$ we observe that, by $(b)$, $Q (e) T(a) = Q (e) T^{**} (e)$, and hence $P_2 (e) T(a) = P_2 (e) T^{**} (e)$. Lemma \ref{l weak-local triple deriv compact b} gives $P_2 (e) T(a) = -Q(e) T(a)$, which proves the desired equality. Finally, $(d)$ follows from $(a)$ and $(c)$.
\end{proof}

We can also derive now a technical variant of the previous result.

\begin{proposition}\label{p technical variation weak-local triple derivations on compact tripotents new} Let $T : E \to E$ be a linear mapping on a JB$^*$-triple satisfying the following hypothesis \begin{enumerate}[$(h1)$]\item $\{a,{T(b)},c\}=0$ for every $a,b,c$ in $E$ with $a, c\perp b$;
\item $P_2 (e) T(a) = -Q(e) T(a)$ for every norm-one element $a$ in $E$, and every tripotent $e$ in $E^{**}$ such that $e\leq s(a)$ in $E_2^{**} (e)$.
\end{enumerate}
Suppose $e$ is a compact tripotent in $E^{**}.$ Then the following statements hold: \begin{enumerate}[$(a)$]
\item $P_0 (e) T^{**}(e) =0$;
\item If $a$ is a norm-one element in $E$ whose support tripotent is $e$ (that is, $e$ is a compact-$G_{\delta}$ tripotent), then $Q (e) T(a) = Q (e) T^{**} (e);$
\item $P_2(e) T^{**}(e) = - Q(e) T^{**}(e);$
\item $T^{**} \{e,e,e\} = 2 \{T^{**} (e),e,e\} + \{e,T^{**} (e), e\}$.
\end{enumerate}
\end{proposition}

\begin{proof} The continuity of $T$ follows, via Proposition \ref{p second necessary condition for weaklocal der implies continuity}, from $(h2)$. Statements $(a)$ and $(b)$ can be derived from Proposition \ref{p local triple derivations on compact tripotents via orthognal forms}. If in the proof of Proposition \ref{p weak-local triple derivations on compact tripotents new} we apply the hypothesis $(h2)$ the same argument runs here.
\end{proof}

The technical arguments given in \cite[proof of Corollary 2.3 and comments prior to it, and proof of Theorem 2.4]{BurPolPerBLMS14} actually prove the following result.

\begin{proposition}\label{p BurPoPe Cor 23 to thmm 24} Let $T: E\to E$ be a bounded linear operator on a JB$^*$-triple satisfying that $T^{**} \{e,e,e\} = 2 \{T^{**} (e),e,e\} + \{e,T^{**} (e), e\}$ for every compact tripotent $e\in E^{**}$. Then $T^{**}$ is a triple derivation.
\end{proposition}

We can now obtain our main result.

\begin{theorem}\label{t weak-local derivations} Let $T:E\to E$ be a linear mapping on a JB$^*$-triple. The following statements are equivalent:
\begin{enumerate}[$(a)$]\item $T$ is a triple derivation;
\item $T$ is a local triple derivation;
\item $T$ is a weak-local triple derivation;
\item  $\{a,{T(b)},c\}=0$ for every $a,b,c$ in $E$ with $a, c\perp b$ and $P_2 (e) T(a) = -Q(e) T(a)$ for every norm-one element $a$ in $E$, and every tripotent $e$ in $E^{**}$ such that $e\leq s(a)$ in $E_2^{**} (e)$.
\end{enumerate}
\end{theorem}

\begin{proof} $(a)\Rightarrow (b) \Rightarrow (c)$ are clear. $(c) \Rightarrow (d)$ is a consequence of Lemmas \ref{l 4 BurFerGarPe} and \ref{l weak-local triple deriv compact b}. Finally, $(d) \Rightarrow (a)$ is a consequence of Propositions \ref{p second necessary condition for weaklocal der implies continuity}, \ref{p technical variation weak-local triple derivations on compact tripotents new} and \ref{p BurPoPe Cor 23 to thmm 24}.
\end{proof}

The following consequence is interesting by itself and it compliments the results in \cite{EssaPeRa16}.

\begin{corollary}\label{t weak-local derivations Cstar algebra} Every weak-local triple derivation on a C$^*$-algebra is a triple derivation.
\end{corollary}

\end{document}